\documentclass[12pt,twoside, reqno]{amsart}

\usepackage{amsmath}
\usepackage{amssymb}
\usepackage{mathrsfs}
\usepackage{units}
\usepackage{url}

\newtheorem{theorem}{Theorem}

\newtheorem{corollary}[theorem]{Corollary}

\theoremstyle{remark}

\newtheorem{remark*}[theorem]{Remark*}

\begin{document}

\title[Zeros  of partial sums of the Riemann zeta-function]
{Zeros of partial sums of the Riemann zeta-function}

\author[Gonek]{S. M. Gonek*}  

\thanks{*Research of the first author was partially supported
by NSF grant   DMS-0653809.}

\address{Department of Mathematics \\
University of Rochester \\
Hylan Hall \\
Rochester, NY 14627, USA}

\email{gonek@math.rochester.edu}
\email{ledoan@math.rochester.edu}
\author[Ledoan]{A. H. Ledoan}



\subjclass[2000]{Primary 11M06, 11M26; Secondary 11M41}



\begin{abstract}
We investigate the distribution of the zeros of partial sums  
of the Riemann zeta-function, $\sum_{n\leq X}n^{-s}$, estimating
the number of zeros up to height $T$, the number of  zeros to
the right of a given vertical line, and other aspects of their horizontal 
distribution.
\end{abstract}

\maketitle

\thispagestyle{empty}


Although a great deal is  known and conjectured about the distribution of zeros 
of the Riemann zeta-function,  
little is known about the  zeros of its partial sums 
\[
F_X(s)=\sum_{n\leq X} n^{-s}.
\]
Here  $s= \sigma+it$  denotes a complex variable  
and    $X$ is at least $2$. 
Exceptions  are the works of P. Tur\'{a}n~\cite{Turan1948}, \cite{Turan1959}, 
\cite{Turan1960}, \cite{Turan1963}, N. Levinson~\cite{Levinson1973}, 
S. M. Voronin~\cite{Voronin1974}, and H. L. Montgomery~\cite{Montgomery1983},
and  numerical studies by R. Spira~\cite{Spira}  
and, more recently, P. Borwein \emph{et al.}~\cite{Bor}.   
Our goal here is to extend these investigations.   

 We write $\rho_X=\beta_X+i\gamma_X$ for a typical  zero  of $F_X(s)$.
 The  number of these up to height $T$ we denote  by $ N_X(T) $, and 
  the number  of these with $\beta_X \geq \sigma$   by $ N_X(\sigma, T)$.  
 We follow the convention that if $T$
is the ordinate of a zero, then $ N_X(T)$, say, is defined as 
$\lim_{\epsilon \to 0^{+}} N_X(T+\epsilon)$.

There are two natural  ways to pose questions about $N_X(T)$, $N_X(\sigma, T)$,
 and the distribution of the zeros generally.
 We can fix an  $X$  and consider zeros with $0<\gamma_X \leq T$  and let $T$ 
 tend  to infinity,  or we can ask for results that are uniform as $X$ and $T$ both  tend  to infinity.  Here we will be concerned with the latter sort of question.

Our first theorem collects together a number of  known results.

\begin{theorem}\label{thm 1}
The zeros of $F_X(s)$ lie in the strip $\alpha<\sigma<\beta$,
where $\alpha$ and $\beta$ are the unique   solutions of the
equations $1+2^{-\sigma}+\dots+(X-1)^{-\sigma}=X^{-\sigma}$
and $2^{-\sigma}+3^{-\sigma}+\dots+X^{-\sigma}=1$, respectively.
In particular, $\alpha>-X$ and $\beta<1.72865$.
For $X$ sufficiently large $F_X(s)$ has
no zeros in the half-plane
\begin{equation*}
\sigma   \geq 1 + \frac{2\log\log X}{\log X}.
\end{equation*}
Moreover, for any constant $c$ with
$c>4/\pi-1$ there exists a number $X_0(c)$ such that if  $X \geq X_0(c)$,
then $F_X(s)$ has at most a finite number of zeros in the half-plane
\begin{equation*}
\sigma  >1 + \frac{c\log\log X}{\log X}.
\end{equation*}
\end{theorem}

\begin{proof}
That the zeros all lie in a strip follows
immediately from the fact that $|1+2^{-s}+\dots+X^{-s}|>0$
if $1+2^{-\sigma}+\dots+(X-1)^{-\sigma}<X^{-\sigma}$ or
if $2^{-\sigma}+\dots+X^{-\sigma}<1$. The  estimates for 
$\alpha$ and $\beta$ may be found in Borwein \emph{et al.}~\cite{Bor}. 
The last two assertions are due to Tur\'{a}n
\cite{Turan1948} and  Montgomery
\cite{Montgomery1983}, respectively.
\end{proof}

\begin{theorem}\label{thm 2}
Let $X, T \geq 2$. Then we have
\begin{equation*}
N_X(T)
  =\frac{T}{2\pi}\log X + O(X).
\end{equation*}
\end{theorem}

Before beginning the proof we note that 
$\Re{F_X(2+it)} \geq 1 - \sum_{2\leq n\leq X} n^{-2} > 0$,
so that $F_X(s)$ has no zeros on the line $\Re s =2$. 
If $t$ is not the ordinate of a zero, we define $\arg F_X(\sigma+it)$
as the value obtained by continuous variation along the straight 
lines joining $2$, $2+it$, and $\sigma+it$, starting with the value $0$.
If $t$ is the ordinate of a zero, we let 
$\arg F_X(\sigma+it) = \lim_{\epsilon \to 0^{+}} 
\arg F_X(\sigma+i(t+\epsilon))$. 
 
\begin{proof}
Let $\mathcal{C}$ be the rectangle with vertices at $-U$,
$2$, $2+iT$, and $-U+iT$, where   $U\geq X$. Clearly
$F_X(s)$ is nonzero on the right and bottom sides of $\mathcal{C}$, 
and by Theorem~\ref{thm 1}  it  does not vanish on the left side.
Without  loss of generality, we may also
assume  $F_X(s)$ is nonzero on the top edge. 
By the  argument principle, 
\begin{equation*}
2\pi N_X(T) = \triangle_{\mathcal{C}}\arg{F_X(s)},
\end{equation*}
where  $\triangle_{\mathcal{C}}$ denotes the change in argument 
around  $\mathcal{C}$ taken  in the positive direction. 
Because $F_X(s)$  is   real  and nonvanishing along $[ -U, 2 ]$,
 the  change in $\arg{F_X(s)}$ on this edge is $0$. 
Since $\Re{F_X(2+it)}>0$ the change in argument along the right 
edge of $\mathcal{C}$ is $\ll 1$.
To estimate the change in  argument along the top edge of $\mathcal{C}$ we 
write  
\begin{equation*}
\Im F_X(\sigma +i T)
  = -\sum_{n\leq X}  \sin(T\log n)  n^{-\sigma}.
\end{equation*}
By a generalization of Descartes' Rule of Signs
(see, for instance,   P\'{o}lya and
Szeg\"{o} \cite{PolyaSzego1971}, Part V, Chapter 1, No. 77),
the number of zeros of $\Im{F_X(s)}$ in the interval
$-U\leq\sigma\leq 2$ is at most the number of changes
of sign in the sequence $\{\sin (T\log n)\}_{n=2}^X$, namely $\ll X$.
Between  consecutive zeros of $\Im{F_X(s)}$, 
$\arg F_X(\sigma+iT)$ changes by at most $\pi$, and it begins
with the value $\arg F_X(2+iT)\ll 1$. Thus  the change in
argument along the top edge is $\ll X$.
 Finally, for  $U \geq X$, 
$X^U> 1 + 2^U + 3^U + \dots + (X-1)^U$,   so
$\triangle \arg{F_X(-U+it)} \big|_{ T}^{0}
= \triangle  \arg (X^{U-it}) \big|_{ T}^{0} +O(1)
=T \log X +O(1)$. Combining  our estimates, we obtain the assertion
 of the theorem.
\end{proof}

Next, we   estimate  
$$
N_X(\sigma,T)=\sum_{\substack{0 <\gamma_X \leq T \\ \beta_X \geq \sigma}} 1
$$
for $\sigma > 1/2$.
We follow one of the many  classical approaches to zero density theorems 
(see Titchmarsh~\cite{Titchmarsh1986}, Theorems 9.16 and 9.17)
and do not strive for the strongest result.


\begin{theorem}\label{thm 3}
Suppose that $X \to \infty$ as $T\to \infty$ and that $X \ll T$.
Then  
\begin{equation*}
N_X(\sigma,T)
=O(TX^{ 1-2\sigma} \log^6 T )
\end{equation*}
uniformly for   $\sigma \geq 1/2+1/\log T$. 
\end{theorem}

\begin{proof}
Let $T\geq 2$ and   define
\begin{equation*}
f_X(s)
  =F_X(s)M_{Y}(s) - 1,
\end{equation*}
where
\begin{equation*}
M_{Y}(s)
  =\sum_{n\leq Y}\mu(n) n^{-s},
\end{equation*}
  $\mu(n)$ is the M\"{o}bius function, and $Y\geq 2$  is to be chosen later as 
  a function of $X$ and $T$.  
We have
\begin{equation*}
f_X(s)
  =\sum_{m\leq X}m^{-s}\sum_{n\leq Y} \mu(n) n^{-s} -1
  =\sum_{X<n\leq XY} a_X(n) n^{-s},
\end{equation*}
where
\begin{equation*}
a_X(n)
  =\sum_{\substack{d\mid n \\ d\leq Y \\ n/d\leq X}}\mu(d).
\end{equation*}
Note that $a_X(n)=1$ if $n=1$, and
$a_X(n)=0$ if $n<X$  or $n>XY$. Furthermore,
$|a_X(n)|\leq d(n)$ for all $n$, where $d(n)$
is the number of   divisors of $n$.

Set
\begin{equation*}
h_X(s)
  =1 - f_X^2(s)
  =F_X(s)M_X(s)\big(2 - F_X(s)M_X(s)\big).
\end{equation*}
Then $h_X(s)$ is holomorphic and vanishes at the zeros of $F_X(s)$.
For $\sigma\geq 2$ and $X$ sufficiently large,
\begin{equation*}\label{ineq on sigma=2}
|f_X(s)|^2
  \leq\bigg(\sum_{X<n\leq XY}\frac{d(n)}{n^2}\bigg)^2
 \ll  X^{-2 }  \log^2 X   <\frac{1}{2}  X^{ \epsilon-2} 
  <\frac{1}{2}.
\end{equation*}
Thus $h_X(s)\neq 0$ for $\sigma \geq 2$ and $X$ large.
Applying Littlewood's lemma to $h_X(s)$
(see, for example, Titchmarsh~\cite{Titchmarsh1986}, Section 9.9), we find that 
if  $\sigma_0 \geq 1/2$,
\begin{equation}\label{Littlewood formula}
\begin{split}
2\pi\sum_{\substack{0\leq\gamma_X\leq T \\ \beta_X>\sigma_0 }}
  (\beta_X-\sigma_0 )
  &\leq \int_0^T\big( \log|h_X(\sigma_0 +it)|
    - \log|h_X(2+it)| \big)\,dt \\
    &\quad + \int_{\sigma_0 }^{2}\big(\arg h_X(\sigma+iT)
    - \arg h_X(\sigma)\big)\,d\sigma.
\end{split}
\end{equation}
Now
\begin{equation*}
\log|h_X(s)|
  \leq\log \big(1 + |f_X(s)|^2 \big)
  \leq|f_X(s)|^2,
\end{equation*}
so we have
\begin{equation}\label{left side}
\begin{split}
\int_0^T\log|h_X(\sigma_0 +it)|\,dt                 
  &\leq\int_0^T|f_X(\sigma_0 +it)|^2\,dt \\
  &=\sum_{X<n\leq XY}\frac{a_X(n)^2}{n^{2\sigma_0 }}\big(T + O(n)\big)  \\
  & \ll T \sum_{X<n\leq XY} \frac{d^2(n)}{n^{2\sigma_0 }}
    + \sum_{X<n\leq XY}\frac{d^2(n)}{n^{2\sigma_0 -1} }     \\
  &\ll  TX^{1-2\sigma_0 } \log^4 T + X^{2-2\sigma_0 }
  (1+Y^{2-2\sigma_0} )\log^4 T.  
\end{split}
\end{equation}
 To obtain  the estimate on the second line we  have used   Montgomery and Vaughan's 
\cite {MontgomeryVaughan1974} mean value theorem 
 for Dirichlet polynomials.
Similarly,
\begin{align}\label{right side}
 \int_0^T\log|h_X(2+it)|\,dt            \notag
& \leq \int_0^T\log|F_X(2+it)|\,dt   \\
 &  \ll T \sum_{X<n\leq XY} \frac{d^2(n)}{n^4}
    + \sum_{X<n\leq XY}\frac{d^2(n)}{n^3 } \\
    & \ll  T X^{-3} \log^3X +  X^{-2} \log^3X .    \notag
\end{align}

By  a well-known lemma in Titchmarsh
\cite{Titchmarsh1986} (see Section 9.4),
 $\arg h_X(s) \ll \log XY$ for
$\sigma \geq 1/2$. Hence
\begin{equation*}
\int_{\sigma_0}^2  \big(\arg h_X(\sigma+iT)
  -\arg h_X(\sigma)  \big) \,d\sigma
\ll  \log XY.
\end{equation*}
Combining this with  the estimates  \eqref{left side} and \eqref{right side}  
in \eqref{Littlewood formula},  we obtain
\begin{equation*}
\begin{split}
\sum_{\substack{0\leq\gamma_X\leq T \\ \beta_X>\sigma_0}}
  (\beta_X-\sigma_0)
   & \ll   TX^{1-2\sigma_0} \log^4 T + X^{2-2\sigma_0}
  (1+Y^{2-2\sigma_0} ) \log^4 T 
   + \log XY.
\end{split}
\end{equation*}
If $\sigma_0 \geq 1$, the first term on the right dominates the others
and, in this case, we let $Y=2$. If $1/2 \leq \sigma_0 <1$, 
we set $Y= C T/X $, where   $C$ is a constant chosen large 
enough to ensure that $Y\geq 2$. The first  term  on 
the right-hand side is again the dominant term, so we see that   
 \begin{equation*}
\sum_{\substack{0\leq\gamma_X\leq T \\ \beta_X>\sigma_0}}
  (\beta_X-\sigma_0)
     \ll  T X^{1-2\sigma_0} \log^4 T.
\end{equation*}

Finally, we set $\sigma_0<\sigma_1 \leq  2$ with
$\sigma_1=\sigma_0 + 1/\log T$ and obtain
\begin{equation*}
(\sigma_1-\sigma_0)N_X(\sigma_1,T)
 \ll TX^{ 1-2\sigma_1}\log^5 T.
\end{equation*}
Therefore $N_X(\sigma_1,T)  \ll TX^{ 1-2\sigma_1}\log^6 T $
uniformly for $\sigma_1 \geq 1/2+1/\log T$.
This completes the proof of the theorem.
\end{proof}


Our next result follows easily from the estimates for $N_X(T)$
and $N_X(\sigma, T)$ in
Theorems~\ref{thm 2} and  \ref{thm 3}.

\begin{corollary}\label{almost all zeros} 
Suppose that $X \to \infty$ as $T\to \infty$ and that $X \ll T$.
There is an absolute constant $c_1$ such that, for
 $T$ sufficiently large,   
$$
\beta_X \leq \frac12 + \frac{c_1\log\log T}{\log X}
$$ 
for almost all  zeros of $F_X(s)$ with 
 $0<\gamma_X \leq T$.
\end{corollary}
 
We can also prove a conditional result in the same vein.

\begin{theorem}
Assume the Riemann Hypothesis.
Suppose that $X \to \infty$ as $T\to \infty$ and that $X \ll T$. 
There exists an absolute  constant $c_2$ such that, for $T$ 
sufficiently large,   
$$
\beta_X \leq \frac12 + \frac{c_2  \log T }{\log X \,\log\log T}
$$
for all zeros of $F_X(s)$ with $X^{1/2} < \gamma_X \leq T$.
\end{theorem}

\begin{proof}
By Theorem 4.1 of Gonek \cite{Gonek2007} there is a
positive constant $A$ such that
\begin{equation*}
\zeta(s)
  =F_X(s) + O\bigg(X^{1/2-\sigma}
   \exp\bigg(\frac{A\log t}{\log\log t}\bigg)\bigg)
\end{equation*}
 for $\sigma$  bounded and $\geq 1/2$, $|s-1|>1/10$, and $9 \leq X \leq t^2$.
By Titchmarsh~\cite{Titchmarsh1986}, equation (14.14.5), 
there is a positive constant $B$ such that
\begin{equation*}
|\zeta(s)|
  \gg \exp\bigg(-\frac{B\log t}{\log\log t}\bigg),
\end{equation*}
for $\sigma \geq 1/2 + B/ \log\log t$. 
It follows that there is a  positive constant $C$ such that
\begin{equation*}
|F_X(s)| >0,
\end{equation*}
when
\begin{equation*}
\sigma 
  >\frac{1}{2}
    + \frac{C}{\log X} \bigg(\frac{ \log t}{\log\log t} \bigg).
\end{equation*}
In light of  the constraint  that $X \leq t^2$, we see that
$F_X(s) \neq 0$ for $X^{1/2} \leq t \leq T $ and 
\begin{equation*}
\sigma 
  >\frac{1}{2}
    + \frac{C\log T}{\log X \log\log T}.
\end{equation*}
The  theorem follows.
\end{proof}


We next show that  zeros to the right of the line $\Re s =1/2$ are 
on average close to it.
 
\begin{theorem}\label{thm dist to 1/2}
For $3 \leq X \leq T$ we have
\begin{equation}
\sum_{\substack{\gamma_X\leq T \\  \beta_X>1/2  }}
\bigg(\beta_X-\frac{1}{2}\bigg)
  \leq\frac{T}{4\pi}\log\log X
   + O\bigg(\frac{X}{\log X}\bigg) .
\end{equation}
\end{theorem}

\begin{proof}
By a straightforward application of Littlewood's lemma, we find that
\begin{equation}\label{Littlewood 10}
2\pi \sum_{\substack{\gamma_X\leq T \\  \beta_X>1/2  }} \bigg(\beta_X-\frac{1}{2}\bigg)
  =\int_0^T\log\bigg|F_X\bigg(\frac{1}{2}+it\bigg)\bigg|\,dt
   + O(\log X).
\end{equation}
Applying the arithmetic-geometric mean inequality
and   Montgomery and Vaughan's
mean value theorem, 
we obtain
\begin{equation*}
\begin{split}
\int_0^T\log\bigg|F_X\bigg(\frac{1}{2}+it\bigg)\bigg|\,dt
  &\leq\frac{T}{2}
    \log\Bigg(\frac{1}{T}
     \int_0^T\bigg|F_X\bigg(\frac{1}{2}+it\bigg)\bigg|^2\,dt\Bigg) \\
  &\leq\frac{T}{2}\log \big(\log X + O(XT^{-1})\big) \\
  &=\frac{T}{2}\log\log X + O\bigg(\frac{X}{\log X}\bigg),
\end{split}
\end{equation*}
where the last line follows because $\log(1+z)\ll z$ for $|z|<1/2$.
Combining this with \eqref{Littlewood 10}, we complete the proof of the theorem.
\end{proof}

As a corollary, we obtain a result that is stronger  than  
Theorem~\ref{thm 3} when (roughly) 
$ (\log T)^{\epsilon-6} \ll \sigma -1/2   \ll   \log\log X/\log X$.   

\begin{corollary}
Let $  X \leq T$. Then for $X$ and $T$ sufficiently large  
\begin{equation*}
N_X\left(\sigma,T\right)
  \leq  \big(1+o(1)\big)   \frac{T\, \log\log X}{4\pi(\sigma-1/2)}.
\end{equation*}
\end{corollary}

\begin{proof}
For $\sigma> 1/2$ we have
\begin{equation*}
\begin{split}
2\pi\sum_{\substack{0<\gamma_X\leq T \\ \beta_X>1/2}}
  \left(\beta_X-\frac{1}{2}\right)
   &\geq 2\pi\sum_{\substack{0<\gamma_X\leq T \\ \beta_X\geq\sigma}}
    \left(\beta_X-\frac{1}{2}\right) \\
   &\geq 2\pi\left(\sigma-\frac{1}{2}\right) N_X(\sigma,T).
\end{split}
\end{equation*}
By Theorem~\ref{thm dist to 1/2},
\begin{equation*}
2\pi\bigg(\sigma-\frac{1}{2}\bigg) N_X(\sigma,T)
  \leq\frac{T}{2}\log\log X + O\bigg(\frac{X}{\log X}\bigg).
\end{equation*}
Hence
\begin{equation*}
N_X(\sigma,T)
  \leq     \big(1+o(1)\big) \frac{T\log\log X}{4\pi(\sigma-1/2)}.
\end{equation*}
\end{proof}

\begin{theorem}\label{Sum beta+U}  
Let $2 \leq X \leq T$.
For $U\geq X$  we have
\begin{equation*}
\sum_{0<\gamma_X\leq T}(\beta_X + U)
  =U \frac{T}{2\pi} \log X + O(UX) + O(T).
\end{equation*}
\end{theorem}

\begin{proof}
Applying Littlewood's lemma to $F_X(s)$, we find that
\begin{equation}\label{Littlewood 20}
\begin{split}
2\pi\sum_{0\leq\gamma_X\leq T}(\beta_X+U)
  &=\int_0^T\big(\log|F_X(-U+it)|
    - \log|F_X(2+it)|\big)\,dt \\
    &\quad + \int_{-U}^{2}\big(  \arg F_X(\sigma+iT)
    - \arg F_X(\sigma)\big)\,d\sigma.
\end{split}
\end{equation}
Because  
$ |F_X(2+it)| $ is bounded above and below by positive constants,  
we see that 
$\int_{0}^{T}  \log |F_X(2+it)| \,dt \ll T$.
Also, as in the proof of Theorem~\ref{thm 2},  
$\arg F_X(\sigma)=0$ and 
$\arg F_X(\sigma + iT) \ll X$ for  $\sigma \leq 2$.
Thus
$\int_{-U}^{2}  \arg F_X(\sigma) \,d\sigma =0$
and
$\int_{-U}^{2}  \arg F_X(\sigma+iT) \,dt \ll UX$.
Finally,   note that
\begin{equation*}
F_X(-U+it) =X^{U-it} \sum_{0\leq n\leq X-1} \bigg(1-\frac{n}{X}\bigg)^{U-it}
\end{equation*}
and
\begin{align*}
  \sum_{1\leq n\leq X-1} \bigg(1-\frac{n}{X}\bigg)^{U}
 <  \int_{0}^{X-1}  \bigg(1-\frac{y}{X}\bigg)^U \,dy  
 <   \frac{X}{U+1} <1.
\end{align*}
Therefore, we find that
$$
\int_0^T \log|F_X(-U+it)| \,dt = \int_0^T   \big(U\log X    +O(1)\big) \,dt
=UT\log X +O(T).
$$

Inserting these estimates  into \eqref{Littlewood 20}, we obtain
$$
2\pi\sum_{0\leq\gamma_X\leq T}(\beta_X+U) =  UT\log X + O(UX) + O(T).
$$
This completes the proof of the theorem.
\end{proof}

\begin{corollary}
Let $2 \leq X \leq T$.
Then  we have
\begin{equation*}
\sum_{0<\gamma_X\leq T}\beta_X 
=O(T)  + O( X^2).
\end{equation*}
In particular, if $X\ll T^{1/2}$ and $X \to \infty$, then 
$$
\frac{1}{ N_X(T)} \sum_{0<\gamma_X\leq T} \beta_X  
=O\bigg(\frac{1}{\log X}\bigg).
$$
That is, the average of the abscissas of the zeros of $F_X(s)$ equals $0$. 
\end{corollary}
\begin{proof}
The  first assertion  follows in a straightforward way from 
Theorems~\ref{thm 1}, ~\ref{thm 2}, and \ref{Sum beta+U}.
The second assertion follows immediately from the first 
and Theorem~\ref{thm 2}.
\end{proof}

Our final result provides   information about  the
zeros of $F_X(s)$ for arbitrary values of $\Re s <1/2$.

\begin{theorem}\label{thm dist left of 1/2}
Let   $2 \leq X \leq T$. Then  uniformly for  $\sigma <1/2$ 
we have
\begin{equation*}
\sum_{\substack{\gamma_X\leq T \\  \beta_X> \sigma  }}
(\beta_X-\sigma)
  \leq      \bigg(\frac12-\sigma \bigg)   \frac{T}{2\pi} \log X
  - \frac{T}{4\pi} \log (1-2\sigma) 
   + O\big( (\sigma+1)X \big) + O(T).
\end{equation*}
\end{theorem}

\noindent {\bf Remark.} 
Taking $\sigma = -X$ in the theorem, we obtain  
\begin{equation*}
\sum_{0<\gamma_X\leq T}(\beta_X + X)
   \leq  X \frac{T}{2\pi} \log X + O(X^2) + O(T).
\end{equation*}
According to Theorem~\ref{Sum beta+U} with $U= X$, we in fact have equality here.
 
\begin{proof}
Let $\sigma_0 <1/2$.
By Littlewood's lemma, 
\begin{equation*} 
\begin{split}
2\pi\sum_{0\leq\gamma_X\leq T}(\beta_X-\sigma_0)
  &=\int_0^T\big(\log|F_X(\sigma_0+it)|
    - \log|F_X(2+it)|\big)\,dt \\
    &\quad + \int_{\sigma_0}^{2}\big(  \arg F_X(\sigma+iT)
    - \arg F_X(\sigma)\big)\,d\sigma.
\end{split}
\end{equation*}
As in the proof of Theorem~\ref{Sum beta+U},    
$\int_{0}^{T}  \log |F_X(2+it)| \,dt \ll T$
and the second integral on the right is $\ll (1+ |\sigma_0|) X $.
Thus
\begin{equation}\label{Littlewood}
2\pi \sum_{\substack{\gamma_X\leq T \\  \beta_X>\sigma_0  }}  
 (\beta_X-\sigma_0  )
  =\int_0^T\log |F_X (\sigma_0+it ) |\,dt
   + O \big((1+|\sigma_0|) X  \big) + O(T).
\end{equation}
Applying the arithmetic-geometric mean inequality
and   Montgomery and Vaughan's
mean value theorem as before, 
we obtain
\begin{equation*}
\begin{split}
\int_0^T\log |F_X (\sigma_0+it  ) |\,dt
  &\leq \frac{T}{2}
    \log \bigg( \frac{1}{T}
     \int_0^T |F_X (\sigma_0+it ) |^2 \,dt \bigg) \\
  &\leq \frac{T}{2}\log \Bigg(\frac{1}{T} 
  \bigg( \sum_{n\leq X} \frac{1}{ n^{2\sigma_0}}                  
   \big(  T+O(n)  \big)  \bigg)   \Bigg) \\
&=  \frac{T}{2}\log\Bigg(  \frac{X^{1-2\sigma_0}}{1-2\sigma_0}  
+O\bigg(\frac{X^{2-2\sigma_0}}{(2-2\sigma_0)T}\bigg)    \Bigg) \\  
&=  \frac{T}{2}\log \Bigg(  \frac{X^{1-2\sigma_0}}{1-2\sigma_0}  
\bigg( 1+O\bigg(\frac{(1-2\sigma_0) X }{(2-2\sigma_0) T}\bigg) \bigg)   \Bigg) \\  
&=\bigg(\frac{1}{2}- \sigma_0\bigg) T\log X   -\frac{T}{2} \log (1- 2\sigma_0) 
 + O(X).
\end{split}
\end{equation*}
Combining this and \eqref{Littlewood}, we obtain the theorem.
\end{proof}




The main   question we have  left unanswered
is whether  one can prove  an asymptotic estimate for
the sum  in Theorem~\ref{thm dist left of 1/2}.
For example, is it the case   that 
$$
\sum_{\substack{\gamma_X\leq T \\  \beta_X> \sigma  }}
(\beta_X-\sigma)
  \sim      \bigg(\frac12-\sigma \bigg)   \frac{T}{2\pi} \log X  
$$
when $\sigma$ is  bounded and less than $1/2$,
and   $X\to\infty$ with $T$?
To answer this would require an asymptotic estimate rather 
than an upper bound for 
 $$
 \int_{0}^{T} \log|F_X(\sigma+it)| dt 
 $$
when $\sigma <1/2$.

\end{document}